\documentclass[reqno]{amsart}
\usepackage{hyperref}

\begin{document}
\title[Negative stiffness]
{Stability of solutions to damped equations with negative stiffness}

\author[J. G. Dix, C. A. Terrero-Escalante]
{Julio G. Dix, C\'esar A. Terrero-Escalante} % In alphabetical order

\address{Department of Mathematics\\
Texas State Univeristy\\
601 University Drive\\
San Marcos, TX 78666 USA}
\email[J. G. Dix]{jd01@txstate.edu}

\address{Departamento de F\'{\i}sica Te{\'o}rica\\
Instituto de F\'{\i}sica\\
Universidade do Estado do Rio de Janeiro\\
Maracan\~a, 20559-900 RJ, Brazil}
\email[C. A. Terrero-Escalante]{cterrero@dft.if.uerj.br}

\begin{abstract}
This article concerns the stability of a model for
mass-spring systems with positive damping and negative
stiffness. It is well known that when the coefficients
are frozen in time the system is unstable. Here we find
conditions on the variable coefficients to prove stability.
In particular, we disprove the believe that if the eigenvalues
of the system change slowly in time the system remains unstable.
We extend some of our results for nonlinear systems.
\end{abstract}

\thanks{Submitted May 30, 2007.}

\maketitle
\numberwithin{equation}{section}
\newtheorem{theorem}{Theorem}[section]
\newtheorem{lemma}[theorem]{Lemma}
\newtheorem{remark}[theorem]{Remark}

\section{Introduction}

In this article, we present conditions for the
stability of solutions to the  differential equation
\begin{gather} \label{e1}
u''(t)+b(t)u'(t)+k(t)u(t)=0,\\
u(t_0)=u_0,\quad u'(t_0)=u_1\,, \label{e2}
\end{gather}
where the coefficient $k$ may have negative values. This equation
has been used for modelling mass-spring systems, where the mass is
one unit, the coefficient $b$ produces a damping effect
proportional to velocity, and the coefficient $k$ is the stiffness
coefficient. Physical examples of systems with negative stiffness
can be found in \cite{wang1,wang2}.

Equation \eqref{e1} is also written in matrix form as
\begin{equation} \label{e3}
\begin{pmatrix}u\\ u'\end{pmatrix}' = A(t) \begin{pmatrix}u\\ u'\end{pmatrix},
\quad A(t) = \begin{pmatrix}0 & 1\\ -k(t) & -b(t) \end{pmatrix}\,.
\end{equation}
Note that the roots of the auxiliary equation of \eqref{e1} and the
eigenvalues of the matrix $A$ are $\lambda_\pm =
\big(-b\pm\sqrt{b^2-4k}\big)/2$.
 For constant coefficients $b$, $k$ with $k<0$, one eigenvalue
is negative and one is positive. This makes the point $u=0$,
$u'=0$ a saddle point and the zero solution unstable (see
definition below).

The literature for this differential equations with
time-varying coefficient has several results about instability
with negative stiffness, but none about stability.
In an attempt to extend the stability results to time-varying
coefficients, the so called \textit{frozen
coefficient method} has been developed. In this technique the
coefficients are frozen in time and the system is analyzed as a
system of constant coefficients \cite[Sec. 10.7]{brogan}.
 Thus arises a belief that if the eigenvalues corresponding to
time-varying coefficients change slowly with respect to time, then
the instability obtained for constant coefficients remains valid.
However,  we did not find a precise statement of how small should
be the rate of change of the eigenvalues. In this article, we
disprove the believed instability  by showing that for each positive number,
there exist coefficients $b$ and $k$ for which \eqref{e1} is stable
and the rate of change in the eigenvalues does not exceed the
given number. See Remark \ref{no-belief}.

The main objective of this article is to  find conditions on
$b(t)$ and $k(t)$ for stability in the negative stiffness case.
More precisely, we find conditions for the transition from
instability (at the frozen state) to stability of systems with
variable coefficients. To this end we use Lyapunov functionals in
sections 2, and a fixed point argument in section 3. Also we show
that for every stiffness coefficient, there is a damping coefficient
that makes \eqref{e1} stable. Similarly, for every non-negative
damping coefficient, we find a stiffness coefficient so that
\eqref{e1} is stable if and only if an integral condition \eqref{fix1}
on $b$ is satisfied. Then as an application, we extend the
stability results to nonlinear systems. We conclude this article
by presenting some instability results that complement those in
the literature.

In this article, we assume that $b(t)$ and $k(t)$
are continuous functions so that standard
arguments in differential equations guarantee the existence
and uniqueness of a solution $u(t)=u(t,u_0,u_1)$.

\noindent\textbf{Definition.} The zero solution  is
\emph{stable} if for each $\epsilon>0$, there exists a
corresponding $\delta(\epsilon,t_0)>0$ such that $u(t_0)^2+u'(t_0)^2<\delta ^2$
implies $u(t)^2+u'(t)^2<\epsilon ^2$ for all $t\geq t_0$.
Equivalently, $\max\{|u(t_0)|,|u'(t_0)|\}<\delta$
implies $\max\{|u(t)|,|u'(t)|\}<\epsilon$.

The zero solution is \emph{asymptotically stable} if for some
$\delta>0$, the condition $u(t_0)^2+u'(t_0)^2<\delta ^2$ implies
 $\lim_{t\to\infty} u(t)^2+u'(t)^2=0$.
Equivalently, $\max\{|u(t_0)|,|u'(t_0)|\}<\delta$
implies $\lim_{t\to\infty}|u(t)|=\lim_{t\to\infty}|u'(t)|=0$.

The zero solution  is strictly  stable if it is stable
and asymptotically stable. A solution that is not stable is called
unstable. For linear systems, the stability of one solution implies
the stability of all solutions, in which case the system is called
stable.

\section{Stability using Lyapunov functionals}

Stability for \eqref{e1}, with positive stiffness, has been established when
$b(t)$ and $k(t)$ are bounded above and below by positive constants in
\cite{starz}. 
Assuming that $|b|$, $|k|$, and $|k'|$ are bounded above,
Ignatiev \cite{ignatiev} proved uniform asymptotic stability, under the assumption
that $k$ and $k'/(2k)+b$ are bounded below by two positive constants.

Stability for systems of the form
$\mathbf{x}'=A(t)\mathbf{x}$ and $\mathbf{x}'=A(t,\mathbf{x})\mathbf{x}$
has been studied by several authors \cite{desoer,hale,napoles,rosen,solo}.
However, their restrictions on the
matrix $A(t)$ do not allow for negative stiffness.
In \cite{hale}, the eigenvalues have negative real part, and the matrix
satisfies some growth conditions.
In \cite{solo} the average of the real part of the eigenvalues is negative
and the matrix satisfies some growth conditions.

Our first stability result reads as follows.

\begin{theorem} \label{lyapunov-stabl0}
The zero solution of \eqref{e1} is stable if, for all $t\geq t_0$, the
following  conditions are satisfied:
\begin{gather}
\label{H1} b(t)>0, \quad \frac1{b(t)}+k(t)\geq M \quad
\text{for a constant $M$ (which maybe negative)},\\
\label{H2} \frac{d}{dt} e^{\frac1{b}+k} \leq -
(e^{\frac1{b}+k}-k)^2/(2b) \,.
\end{gather}
\end{theorem}

\begin{proof}
First, we define the Lyapunov functional
\[
E(t)=e^{\frac1{b}+k} u^2+(u')^2
\]
and compute its derivative along the solutions of \eqref{e1},
\[
E'(t)=\frac{d}{dt} [e^{\frac1{b}+k}] u^2 + 2(e^{\frac1{b}+k}- k)uu'
-2b(u')^2\,.
\]
Then, factoring $-2b$ in the last two terms, and completing the
square, we have
\begin{equation} \label{E'}
E'(t)=\big[\frac1{2b}(e^{\frac1{b}+k} -k)^2
+ \frac{d}{dt} e^{\frac1{b}+k}\big]u^2
-2b\big[\frac1{2b}(e^{\frac1{b}+k} -k)u-u'\big]^2\,.
\end{equation}
By \eqref{H2}, the coefficient of $u^2$ is non-positive, and because $b>0$,
$E'(t)\leq 0$; therefore,  $E(t)\leq E(t_0)$
for all $t\geq t_0$.

Now, we show that the zero solution is stable.
For $\epsilon>0$, we select $\delta>0$ such that
\[
\delta^2 < \min\{e^M,1\}\epsilon^2/\max\{1,\exp(\frac1{b(t_0)}+k(t_0))\}\,.
\]
Note that $u(t_0)^2+u'(t_0)^2<\delta^2$ implies
\[
\min\{e^M,1\}\epsilon^2 > \max\{1,\exp(\frac1{b(t_0)}+k(t_0))\}\delta^2
\geq E(t_0)\,.
\]
Also note that
\[
E(t_0)\geq E(t)\geq e^Mu(t)^2+u'(t)^2 \geq
\min\{e^M,1\}\big(u(t)^2+u'(t)^2\big)\,.
\]
The stability of the zero solution follows from the two
inequalities above. This completes the proof.
\end{proof}

Uniform stability is obtained under the additional assumption that
$\frac1{b}+k$ is bounded above; because, the delta in the proof
can be chosen independent of $t_0$.

\begin{remark} \rm
The above theorem makes the transition from instability (at the
frozen state) to stability take place. However, when $k(t)$ is
non-positive and non-decreasing, Conditions \eqref{H1} and \eqref{H2}
imply $b(t)$ growing exponentially, which is very restrictive.
To prove this remark, note that for non-decreasing $k$, we have
$(1/b)'\leq (1/b+k)'$.
Then by \eqref{H1}, $e^M\leq e^{\frac1{b}+k}$ and
$e^{\frac1{b}+k}-k\geq e^M$. Then by \eqref{H2}, $(1/b)'e^M \leq
-\frac{1}{2b} e^M$. This implies $b'\geq b/2$, which in turn
implies $b(t)\geq b(0)e^{t/2}$.
\end{remark}

\noindent\textbf{Example.} Among the equations with exponential damping,
there are stable equations that satisfy and some others that do not
satisfy \eqref{H1}-\eqref{H2}.
For instance, if $k=-1$ and $b=e^{n t}$, then \eqref{e1} has
solutions of the form
\[
\exp\big(\frac {n^2 t - e^{n t}}{2n} \big)
\Big( c_1\big[
    \mathcal{I}(-\frac 12 + \frac 1n,z)
   +\mathcal{I}(\frac 12 + \frac 1n,z)  \big]
+c_2\big[\mathcal{K}(\frac 12 - \frac 1n,z)
       -\mathcal{K}(\frac 12 + \frac 1n,z)\big] \Big)\, .
\]
where $c_1$ and $c_2$ are the integration constants, $z\equiv
{e^{n t}}/{2n}$, and $\mathcal{I}(a,z)$ and $\mathcal{B}(a,z)$ are
the modified Bessel functions of first and second kind,
respectively. When $n=1$ this solution converges uniformly to
$c_1$, when $n>1$ it converges to $2 c_1 \sqrt{n/\pi}$, as
$t\to\infty$. In both cases this implies stability. The zero
solution is also stable for $n\leq 1/2$. For instance, for $n=
1/2$ the solution converges to $32 c_1$, and for $n= 1/4$ it does
to $781250 c_1$. However, \eqref{H1}-\eqref{H2} are satisfied only
for $n\geq 2$.

In an attempt to weaken the growth restrictions on $b$, we
consider now the case where the damping is a positive constant.
Note that the assumptions below restrict the stiffness to remain
negative. Also note that the larger (smaller) the constant damping
is, the slower (faster) the negative stiffness is needed to be
compensated.

\begin{theorem} \label{bconst}
The zero solution of \eqref{e1} is stable if $b(t)$ is a positive
constant and for all $t\geq t_0$:
\begin{gather}
\text{There exists a positive constant $\alpha$ such that }
-k(t) \geq \alpha \label{H1'},\\
-k'+2k^2/b \leq 0,. \label{H2'}
\end{gather}
\end{theorem}

\begin{proof}
We define the Lyapunov functional
\[
E(t)=-k(t)u(t)^2 + u'(t)^2,
\]
whose derivative along solutions of \eqref{e1} is
\begin{align*}
E'(t)&=-k' u^2 -2kuu' + 2u'u''\\
%&=-k' u^2 -2kuu' + 2u'(-bu'-ku)\\
&=-k' u^2 -4kuu' -2b(u')^2\\
&=[-k'+2\frac{k^2}b]u^2 - 2b[\frac{k}b u + u']^2.
\end{align*}
By \eqref{H2'}, the coefficient of $u^2$ is non-negative.
Since $b>0$, $E'(t)\leq 0$ so that $E(t)\leq E(t_0)$ for all $t\geq t_0$.
To show stability of the zero solution, for each positive $\epsilon$,
we select $\delta>0$ such that
\[
\max\{1,-k(t_0)\}\delta^2 < \min\{1,\alpha\}\epsilon^2.
\]
With this delta, we can show that the definition of stability is satisfied,
and hence the proof is complete.
\end{proof}

Note that \eqref{H2'} implies $k'\geq 2k^2/b$ which yields
a lower bound for the rate of change in $k$. Since $b$ is constant, this
inequality provides bounds for the rate of change in the
eigenvalues of $A(t)$:
\[
\frac{d \lambda_+}{dt} \leq -\frac{2k^2}{b\sqrt{b^2-4k}}\leq0
\quad\text{and}\quad
\frac{d \lambda_-}{dt} \geq \frac{2k^2}{b\sqrt{b^2-4k}}\geq0\,.
\]
So that the transition to stability happens when both eigenvalues
approach zero sufficiently fast.

\section{Stability using a fixed point theorem}

In this section we eliminate the restriction
that $b$ must grow exponentially, by using a fixed point argument
similar to those used in \cite{burton}.
We start by stating a condition that is necessary (but not sufficient)
for stability.

\begin{lemma} \label{lemness}
Assume $k(t)\leq 0$ for $t\geq t_0$. Then the condition
\begin{equation}
\int_{t_0}^\infty e^{-\int_{t_0}^s b} ds< \infty \label{fix1}
\end{equation}
is necessary for stability of \eqref{e1}.
\end{lemma}

\begin{proof}
First, using the integrating factor $\exp\big(\int_{t_0}^t b(s)\,ds\big)$
we transform \eqref{e1} into the equivalent equation
\begin{equation} \label{e1e}
\big(e^{\int b}u'\big)'+k e^{\int b} u=0
\end{equation}
which is used for setting up a contrapositive argument.
Let the initial values $u(t_0)$ and $u'(t_0)$ be positive.
Then by the continuity of the solution there is a non-empty
maximal interval $[t_0,t_1]$ where $u(t)\geq 0$. On this
interval,
\[
\big(e^{\int b}u'\big)'= -ke^{\int b} u \geq 0.
\]
So that $e^{\int b}u'$ is non-decreasing; hence,
\[
u'(t) \geq u'(t_0) e^{-\int_{t_0}^t b} >0\,.
\]
Therefore, $u(t)$ is increasing and the maximal interval
can be extended to $[t_0,\infty)$. Integration on the above
inequality yields
\[
u(t)\geq u(t_0)+u'(t_0)\int_{t_0}^t e^{-\int_{t_0}^s b}ds, \quad
\forall t\geq t_0\,.
\]
Note that when \eqref{fix1} is not satisfied, the solution
$u(t)$ is unbounded which implies  \eqref{e1} being unstable.
This completes the proof.
\end{proof}

\begin{remark} \label{specialk} \rm
For each damping coefficient $b(t)$, there is a stiffness
coefficient $k$ so that \eqref{e1} is stable if and only if
\eqref{fix1} is satisfied. In fact,
$$
k(t)=-\exp\big(-2\int_{t_0}^t b(s) ds\big)
$$
 leads to  $\exp(\pm \int_{t_0}^t e^{-\int_{t_0}^s
b}\,ds)$ being solutions of \eqref{e1}. To check stability, we use
that all solutions can be written as $u(t)=c_1e^r+c_2e^{-r}$ with
$r(t)=\int_{t_0}^t e^{-\int_{t_0}^s b}\,ds$, and that by
\eqref{fix1}, $r(t)$ and $r'(t)$ are bounded. On the other hand if
\eqref{fix1} is not satisfied, one of the two solutions is
unbounded which leads to instability.

To check that the two functions above are solutions, let
$\phi(t) =e^{r(t)}$. Then $\phi'=r'e^r$  and $\phi''=(r')^2e^r+r''e^r$.
Since $r'=e^{-\int b}$, $(r')^2=e^{-2\int b}=-k$,
and $r''=-b e^{-\int b}$, it follows that
 $k\phi+(r')^2e^r=0$ and $b\phi'+r''e^r=0$. Therefore,
$\phi$ is a solution of \eqref{e1}.
\end{remark}

As an illustration of the remark above, we have the following two examples:
Firstly, when $b(t)=2/(t+1)$,  \eqref{fix1} is satisfied and the equation
\[
u''(t)+\frac 2{t+1}u'(t)-\frac 1{(t+1)^4}u(t)=0
\]
has solutions of the form $u(t)= c_1 e^{1/(t+1)} +c_2
e^{-1/(t+1)}$. Since both exponential functions are bounded on
$[0,\infty)$, we can prove stability. Secondly, when
$b(t)=1/(t+1)$, \eqref{fix1} is not satisfied and the equation
\[
u''(t)+\frac 1{t+1}u'(t)-\frac 1{(t+1)^2}u(t)=0
\]
has solutions of the form
$u(t)= c_1 (t+1) +c_2 /(t+1)$.
Since the first  function is unbounded on $[0,\infty)$,
we have instability.

Now we set up a mapping whose fixed points are solutions of \eqref{e1}.
 From \eqref{e1e}, it follows that the solution $u(t)$ satisfies
\[
u'(t)=u'(t_0)e^{-\int_{t_0}^t b} - e^{-\int_{t_0}^t b}
\int_{t_0}^t k(\tau) e^{\int_{t_0}^\tau b}
u(\tau)\,d\tau\,,
\]
and
\begin{equation}
u(t)=u(t_0)+u'(t_0)\int_{t_0}^t e^{-\int_{t_0}^s b}\,ds
-\int_{t_0}^t   e^{-\int_{t_0}^s b} \int_{t_0}^s k(\tau) e^{\int_{t_0}^\tau b}
u(\tau)\,d\tau\,ds
:= F[u](t)
\label{eF}
\end{equation}
For $\epsilon>0$, we define the following convex subset of the space of
continuous differentiable functions. Let
\begin{align*}
B_\epsilon&=\big\{ u: |u(t_0)|\leq \epsilon/4,\;
|u'(t_0)|\leq \min\{\epsilon/4,\epsilon/(4\int_{t_0}^\infty e^{-\int_{t_0}^\tau
 b} d\tau)\},\\
&\quad |u(t)|\leq \epsilon,\; |u'(t)|\leq \epsilon \; \forall t\geq t_0\big\}\,.
\end{align*}
Using the supremum norm $\|u\|_\infty=\sup_{t\geq t_0}|u(t)|$, we can
show that this set is closed under the norm
$\|u\|_\infty+\|u'\|_\infty$.
Note that under assumption \eqref{fix1} this set is not empty;
at least, there are constant functions in this set.

For the next result we define the hypotheses:
\begin{gather}
 e^{-\int_{t_0}^s b} < 2 \quad \forall s\geq t_0, \label{fix2}
\\
\int_{t_0}^\infty   e^{-\int_{t_0}^s b} \int_{t_0}^s |k(\tau)|
 e^{\int_{t_0}^\tau b} d\tau\,ds < \frac 12, \label{fix3}
\\
e^{-\int_{t_0}^s b} \int_{t_0}^s |k(\tau)|
 e^{\int_{t_0}^\tau b} d\tau < \frac 12 \quad \forall s\geq t_0\,.  \label{fix4}
\end{gather}

\begin{lemma} \label{fixed-point}
Under assumptions \eqref{fix1} and \eqref{fix2}--\eqref{fix4}, the
transformation  $F$ maps
$B_\epsilon$ into $B_\epsilon$ and has a fixed point.
\end{lemma}

\begin{proof}
Let $u$ be a function in $B_\epsilon$. Then by \eqref{fix2} and \eqref{fix4},
\[
|F[u]'(t)|\leq |u'(t_0)|e^{-\int_{t_0}^t b}
+\|u\|_\infty  e^{-\int_{t_0}^t b}
\int_{t_0}^t |k(\tau)| e^{\int_{t_0}^\tau b}\,d\tau
< \epsilon/2 +\epsilon/2 = \epsilon\,.
\]
By \eqref{fix3},
\begin{align*}
|F[u](t)|&\leq |u(t_0)|+|u'(t_0)|\int_{t_0}^t e^{-\int_{t_0}^t b}
+\|u\|_\infty \int_{t_0}^t e^{-\int_{t_0}^s b}
\int_{t_0}^s |k(\tau)| e^{\int_{t_0}^\tau b} \,d\tau\,ds \\
&< \epsilon/4 +\epsilon/4 +\epsilon/2 =\epsilon\,.
\end{align*}
Therefore, $F$ maps $B_\epsilon$ into itself. To find a fixed
point for $F$, we define an iterative process that can start at any
function $w_0$ in $B_\epsilon$. For $n=1,2,\dots$, define
$w_n=F[w_{n-1}]$. Note that the values $w_n(t_0)$ and $w_n'(t_0)$
remain unchanged in these iterations. Also note that by
\eqref{fix3} and \eqref{fix4},
\[
|(w_{n+1}-w_n)(t)| \leq \int_{t_0}^t e^{-\int_{t_0}^s b}
\int_{t_0}^s |k(\tau)| e^{\int_{t_0}^\tau b}
\|w_n-w_{n-1}\|_\infty \,d\tau\,ds
<\frac 12 \|w_n-w_{n-1}\|_\infty
\]
and
\[
|(w_{n+1}'-w_n')(t)| \leq  e^{-\int_{t_0}^t b}
\int_{t_0}^t |k(\tau)| e^{\int_{t_0}^\tau b}
\|w_n-w_{n-1}\|_\infty \,d\tau
<\frac 12 \|w_n-w_{n-1}\|_\infty\,.
\]
Therefore, $F$ is a contraction and $\{w_n\}$ converges to a fixed
point of $F$ in $B_\epsilon$; hence the solution of \eqref{e1e} is
in the set $B_\epsilon$ and satisfies the conditions for
stability.
 This completes the proof.
\end{proof}

We are ready to present the main result of this section.

\begin{theorem} \label{fix-stable}
Under assumptions \eqref{fix1} and \eqref{fix2}--\eqref{fix4},
the zero solution of \eqref{e1} is stable.
\end{theorem}

\begin{proof} For each $\epsilon>0$, we define $B_\epsilon$ as above
and set $\delta=\min\{\epsilon/4,\epsilon/(4\int_{t_0}^\infty e^{-\int b})\}$.
For initial conditions $|u(t_0)|<\delta$ and  $|u'(t_0)|<\delta$, the
solution is obtained as a fixed point of $F$; therefore,
$|u(t)|<\epsilon$ and  $|u'(t)|<\epsilon$ which implies stability
of \eqref{e1}.
\end{proof}

\begin{remark} \label{specialb} \rm
For each stiffness coefficient  $k(t)$, there exists a damping
coefficient $b(t)$ that makes \eqref{e1} stable. In fact, for a
constant $\alpha>1$, we let
\begin{equation} \label{bcondition}
b(t)\geq 2|k(t)|(t+\alpha)^2 + 2/(t+\alpha)\,,
\end{equation}
so that the conditions in Theorem \ref{fix-stable} are satisfied
with $t_0=0$. Condition \eqref{fix2} follows from $b(t)\geq 0$.
Note that $b\geq 2/(t+\alpha)$ and $\int_0^t b \geq \int_0^t
2/(t+\alpha) =2\ln(t+\alpha)$. So that
\[
\exp(-\int_0^t b )\leq \exp(-2\ln(t+1\alpha))=(t+\alpha)^{-2}.
\]
Condition \eqref{fix1} follows from integrating in the inequality
above. Note that from \eqref{bcondition},
 $|k| \leq (t+\alpha)^{-2}b/2 -(t+\alpha)^{-3}$ which implies
\[
|k|e^{\int b}  \leq \frac12 (t+\alpha)^{-2}b e^{\int b}
-(t+\alpha)^{-3}e^{\int b}
\]
Integrating on $[0,t]$, we have
\[
\int_0^t |k|e^{\int b} d\tau \leq  \frac12 (t+\alpha)^{-2} e^{\int b} -\frac{\alpha}{2}.
\]
Then
\[
e^{-\int b}\int_0^t |k|e^{\int b} d\tau < \frac12 (t+\alpha)^{-2} \leq \frac12
\]
which is \eqref{fix4}. Integrating on $[0,\infty)$,
we have
\[
\int_0^\infty e^{-\int b}\int_0^t |k|e^{\int b} d\tau dt <
\frac12 \int_0^\infty (t+\alpha)^{-2} dt = \frac1{2\alpha^2}\leq \frac12
\]
which is \eqref{fix3}. Therefore, \eqref{e1} is stable with this
choice of $b(t)$.
\end{remark}

\begin{remark} \label{no-belief} \rm
The believed instability is disproved as follows: For each $\epsilon>0$,
we find $b(t)$ and $k(t)$ such that  (\ref{e1}) is stable and the rate of
change in the eigenvalues of $A(t)$ is less than $\epsilon$,
in absolute value.

Let $b(t)=4/(t+\alpha)$ and $k(t)=-1/(t+\alpha)^3$, where
$\alpha=\max\{1,\sqrt{5/\epsilon}\}$.
Note that \eqref{bcondition} is satisfied, and hence the conditions
for Theorem \ref{fix-stable} are satisfied; so that \eqref{e1} is stable.

The rate of change in the eigenvalues of $A(t)$ is
\[
\frac{d\lambda_\pm}{dt}
= \frac 12\Big(-1 \pm \frac b{\sqrt{b^2-4k}}\Big)\frac{db}{dt}
\mp \frac 1{\sqrt{b^2-4k}}\frac{dk}{dt}\,.
\]
The  absolute value of the 
coefficient of $db/dt$ is bounded by $1$, while
$|db/dt|=4/(t+\alpha)^2\leq 4/\alpha^2$. The coefficient of
$dk/dt$ is bounded as follows
\[
\big|\frac{1}{\sqrt{b^2-4k}}\big|
= \big|\frac{1}{b^2-4k}\big|^{1/2}
\leq \frac 1b
= \frac{(t+\alpha)}4.
\]
Since $dk/dt=3/(t+\alpha)^4$, and $\alpha\geq 1$,
\[
|\frac{d\lambda_\pm}{dt}|
\leq \frac{4}{\alpha^2}+ \frac{3}{4(t+\alpha)^3}
\leq \frac{4}{\alpha^2}+ \frac{3}{4\alpha^3}
\leq \frac{19}{4\alpha^2}<\epsilon
\]
Which proves the claim of this remark.
\end{remark}

We conclude this section with a stability result for the non-linear
case.

\begin{remark} \label{nonlinear} \rm
As an applications to non-linear equations, we consider
the differential equation
\begin{equation} \label{e10}
u''(t)=f(t,u',u),
\end{equation}
where $f(t,0,0)=0$ and $f$ is differentiable at $(t,0,0)$.
Then $u(t)\equiv 0$ is a solution of \eqref{e10}. The stability
of the zero solution is studied by
considering the linearized version
\begin{equation}
u''(t)= f_2(t,0,0)u'+f_3(t,0,0)u\,, \label{e11}
\end{equation}
where $f_2(t,x,y)=\partial_x f(t,x,y)$ and
$f_3(t,x,y)=\partial_y f(t,x,y)$.
Note that this linear approximation is valid only for small values
of $u$ and of $u'$.
The stability of \eqref{e11} is studied by setting
\[
b(t)=-f_2(t,0,0), \quad k(t)=-f_3(t,0,0)
\]
and applying results from this section, without any further
modifications. However, the instability results in the next section
may not hold because the linear approximation is valid
only for values $u$, $u'$ close to zero.
\end{remark}

\section{Instability using Lyapunov and Chetaev functionals}

Instability of \eqref{e1}, with negative stiffness, was obtained by
Ignatiev \cite{ignatiev}, assuming
that $|b|$, $|k|$, $|k'|$ are bounded above, and that $-k$ and
$|\frac12 \frac{k'}{k}+b|$ are bounded below by two positive constants.
In the same article, instability is proved when
$\frac14 b^2+k\leq 0$, and when $\frac14 b^2+k> 0$ with some additional
assumptions.

Our next result states that the zero solution is unstable, for non-positive
stiffness and non-positive damping.

\begin{theorem} \label{shetaev1}
If $b(t)\leq 0$ and $k(t)\leq 0$ for all $t\geq t_0$, then
the zero solution is unstable.
\end{theorem}

\begin{proof}
For each pair of initial values $u(t_0)>0$ and $u'(t_0)>0$, by continuity
of the solution and its derivative, there exists an interval
where $u(t)\geq 0$ and $u'(t)\geq 0$. Since $b\leq 0$ and $k\leq 0$,
from \eqref{e1}, it follows that $u''(t)\geq 0$;
i.e., $u$ is concave up and  $u'(t)\geq u'(t_0)>0$ on this  interval.
Let $t_1$ be the largest value such that $u''(t)\geq 0$ for $t_0\leq t \leq t_1$.
If $t_1<+\infty$, the graph of $u$ is concave up
on $[t_0,t_1]$, $u(t_1)\geq u(t_0)>0$ and $u'(t_1)\geq u'(t_0)>0$.
By continuity of the solution and its derivative, there exists $t_2>t_1$,
 such that
$u(t)\geq 0$ and $u'(t)\geq 0$ on $[t_1,t_2]$. Since $b\leq 0$ and $k\leq 0$,
from \eqref{e1}, it follows that $u''(t)\geq 0$ on $[t_1,t_2]$.
This contradicts $t_1$ being maximal; therefore, $t_1=+\infty$.

Because $t_1=+\infty$, the graph of $u$ is concave up
and  $u'(t)\geq u'(t_0)>0$ for all $t\geq t_0$.
Therefore, $\lim_{t\to\infty}u(t)=+\infty$ for all arbitrarily
small and positive initial values. This implies instability of the
zero solution and completes the proof.
\end{proof}

\noindent\textbf{Example.} The equation
\begin{equation}
u''(t)-\frac 1t u'(t)-\frac 1t u(t)=0
\label{examNegBK}
\end{equation}
satisfies the conditions
of Theorem \ref{shetaev1}
and has solutions of the form
\[
u(t)= c_1t\mathcal{I} ( 2,2\sqrt{t}) + c_2t \mathcal{K}( 2,2\sqrt{t} )\, .
\]
For large $t$ the dominant terms in each branch of this solution are
\[
\frac{c_1}{2\sqrt {\pi }}\,t^\frac 34\, e^{2\sqrt{t}}
+ \frac{c_2\sqrt {\pi }}{2}\,t^\frac 34\, e^{-2\sqrt{t}}\, .
\]
which diverges as $t\to\infty$, which implies instability.

\begin{theorem} \label{shetaev2}
The zero solution of \eqref{e1} is unstable under the following
conditions: For all $t\geq t_0$,
\begin{gather}
\label{H5} \text{there exists a positive constant $\alpha$ such that }
-k(t)\geq \alpha\,, \\
\label{H6} k'(t)+2b(t)k(t)\geq 0\,.
\end{gather}
\end{theorem}

\begin{proof}
We construct a functional similar to the one in Chetaev's theorem
\cite{khalil}. 
However, the proof for the variable coefficient case is not the
same as the constant case. Let
\[
V(t)= u(t)^2+\frac{1}{k(t)} u'(t)^2,
\]
whose derivative along the solutions of \eqref{e1} is
\begin{equation} \label{V'}
V'(t)=2uu'-\frac{k'}{k^2}(u')^2+\frac{2}{k}u'u''
     =-\frac{k'+2kb}{k^2}(u')^2\,.
\end{equation}
Note that $V(t)$ can be positive or negative and that
 we can select $u(t_0)$ and $u'(t_0)$ so that $V(t_0)<0$.
Then by \eqref{H6}, $V'(t)\leq 0$ and $V(t)$ is non-increasing,
which allows only
two possible cases:\\
Case 1. When $\lim_{t\to\infty}V(t)=-\infty$, since
$V\geq \frac1{k}(u')^2\geq -\frac1{\alpha}(u')^2$,
we have $\lim_{t\to\infty}(u')^2=+\infty$.
Therefore,  $u'$ is unbounded and the zero
solution is unstable.
\\
Case 2. When $V(t)$ is bounded below, being non-increasing,
it converges to some negative number $-L^2$. Then there exists $t_1$
such that $V(t)\leq -L^2/4$ for $t\geq t_1$. By \eqref{H5},
\[
-\frac1{\alpha}(u')^2 \leq \frac1{k(t)}(u')^2 \leq V(t)\leq -L^2/4
\quad \text{for } t\geq t_1,
\]
which implies $|u'(t)|\geq L\sqrt{\alpha}/2>0$.
If $u'$ is positive, then it is bounded  below by a positive constant
for $t\geq t_1$. Therefore, $\lim_{t\to\infty}u(t)=+\infty$
and the zero solution is unstable.
If $u'$ is negative, then it is bounded above by a negative constant
for $t\geq t_1$. Therefore, $\lim_{t\to\infty}u(t)=-\infty$
and the zero solution is unstable.
This completes the proof.
\end{proof}

\begin{theorem} \label{unst3}
The zero solution of \eqref{e1} is unstable under the following
conditions: For all $t\geq t_0$,
\begin{gather}
 \text{there exists a positive constant $\alpha$ such that }
-k(t)\geq \alpha, \label{H7} \\
 k'(t)+2b(t)k(t)\leq 0\,,  \label{H8} \\
\begin{gathered}
\text{there exist positive constants $\alpha_3$ and $t_3$ such that} \\
-k'(t)-2b(t)k(t)\geq \frac 1t \alpha_3 b(t)^2k(t)^2, \quad \text{for }
t\geq t_3\,.
\end{gathered} \label{H9}
\end{gather}
\end{theorem}

\begin{proof}
We define the Chetaev's functional and compute its derivative as
in the proof Theorem \ref{shetaev2}.
Next we select $u(t_0)$ and $u'(t_0)$ so that $V(t_0)>0$.
Then by \eqref{H8}, $V'(t)\geq 0$ and $V(t)$ is non-decreasing,
which allows only two possible cases:

Case 1. When $\lim_{t\to\infty}V(t)=+\infty$,
since $V(t)\leq u(t)^2$, we have $\lim_{t\to\infty}u(t)=+\infty$
and the zero solution is unstable.

Case 2. When $V(t)$ is bounded above, being nondecreasing,
it converges to some positive number $L^2$. Therefore,
$\int_{t_0}^\infty V'(t)\,dt$ converges.
Since the integral $\int_{t_0}^\infty 1/t\,dt$ diverges, a limit comparison
yields $\lim_{t\to\infty} tV'(t)=0$.
 From \eqref{V'}, \eqref{H7}, and \eqref{H9}, we have
\[
0=\lim_{t\to\infty} tV'(t)
  \geq \lim_{t\to\infty} \alpha_3 b(t)^2 (u'(t))^2 \geq 0\,.
\]
Then  $\lim_{t\to\infty}|bu'|=0$; therefore, there exists a time $t_1$
such that $|b(t)u'(t)|\leq L\alpha/4$ for all $t\geq t_1$.

Since $V(t)\leq u(t)^2$ and $\lim_{t\to\infty}V(t)=L^2$, there exists
a time $t_2$ such that $|u(t)|\geq 3L/4$ for all $t\geq t_2$.
There are two possible cases for $t\geq\max\{t_1,t_2\}$:

 Case 1: $u(t)\geq 3L/4$. From \eqref{e1} and \eqref{H7},
\[
u''(t)=-bu'-ku
\geq -\frac{L\alpha}{4}-\frac{k3L}{4}
\geq -\frac{L\alpha}{4}+\frac{3\alpha L}{4}
= \frac{L\alpha}{2}>0.
\]
The solution $u$, being bounded below and having concavity greater
than a positive constant, must have $\lim_{t\to\infty}u(t)=+\infty$.
This implies instability of the zero solution.

 Case 2: $u(t)\leq -3L/4$. From \eqref{e1} and \eqref{H7},
\[
u''(t)=-bu'-ku
\leq \frac{L\alpha}{4}+\frac{k3L}{4}
\leq  \frac{L\alpha}{4} -\frac{3\alpha L}{4}
=-\frac{L\alpha}{2}<0.
\]
The solution $u$, being bounded above and having concavity less than a
negative constant, must have $\lim_{t\to\infty}u(t)=-\infty$.
This implies instability of the zero solution.

Since the above reasoning applies to arbitrarily small  positive initial
conditions such that $-k(t_0)u(t_0)^2\geq u'(t_0)^2$, the zero solution is
unstable and the proof is complete.
\end{proof}

\subsection*{Aknowledgements}
C.A.T-E. wants to thank, for the kind hospitality, the  MCTP at the 
University of Michigan, Ann Arbor,
where part of this work was done.

\end{document}